\documentclass[preprint,12pt]{elsarticle}

\usepackage{amssymb}
 \usepackage{graphicx}
\usepackage{epsfig}

\usepackage{amsthm}

\newcommand{\sysn}{\left\{\begin{array}{rcl}}
\newcommand{\sysk}{\end{array}\right.}

\newtheorem{theorem}{Theorem}[section]
\newtheorem{lemma}[theorem]{Lemma}

\theoremstyle{example}

\theoremstyle{definition}

\journal{...}

\begin{document}

\begin{frontmatter}

\title{A functional characterization of the Hurewicz property}


\author{Alexander V. Osipov}

\ead{OAB@list.ru}

%\tnotetext[label1]{The research has been supported by .}

\address{Krasovskii Institute of Mathematics and Mechanics, Ural Federal
 University,

 Ural State University of Economics, Yekaterinburg, Russia}

\begin{abstract} For a Tychonoff space $X$, we denote by $C_p(X)$
the space of all real-valued continuous functions on $X$ with the
topology of pointwise convergence.  We give the functional
characterization of the covering property of Hurewicz.

\end{abstract}


\end{frontmatter}

% \linenumbers


\section{Introduction}

Many topological properties are defined or characterized in terms
 of the following classical selection principles.
 Let $\mathcal{A}$ and $\mathcal{B}$ be sets consisting of
families of subsets of an infinite set $X$. Then:

$S_{1}(\mathcal{A},\mathcal{B})$ is the selection hypothesis: for
each sequence $(A_{n}: n\in \mathbb{N})$ of elements of
$\mathcal{A}$ there is a sequence $\{b_{n}\}_{n\in \mathbb{N}}$
such that for each $n$, $b_{n}\in A_{n}$, and $\{b_{n}:
n\in\mathbb{N} \}$ is an element of $\mathcal{B}$.

$S_{fin}(\mathcal{A},\mathcal{B})$ is the selection hypothesis:
for each sequence $(A_{n}: n\in \mathbb{N})$ of elements of
$\mathcal{A}$ there is a sequence $\{B_{n}\}_{n\in \mathbb{N}}$ of
finite sets such that for each $n$, $B_{n}\subseteq A_{n}$, and
$\bigcup_{n\in\mathbb{N}}B_{n}\in\mathcal{B}$.

$U_{fin}(\mathcal{A},\mathcal{B})$ is the selection hypothesis:
whenever $\mathcal{U}_1$, $\mathcal{U}_2, ... \in \mathcal{A}$ and
none contains a finite subcover, there are finite sets
$\mathcal{F}_n\subseteq \mathcal{U}_n$, $n\in \mathbb{N}$, such
that $\{\bigcup \mathcal{F}_n : n\in \mathbb{N}\}\in \mathcal{B}$.

The papers \cite{jmss,ko,sch3,sch1,tss1,bts,ts2} have initiated
the simultaneous
 consideration of these properties in the case where $\mathcal{A}$ and
 $\mathcal{B}$ are important families of open covers of a
 topological space $X$.

In papers [1-8, 12-32]
 (and many others) were investigated the applications of selection principles in the
study of the properties of function spaces. In particular, the
properties of the space $C_p(X)$ were investigated. In this paper
we continue to study different selectors for sequences of dense
sets of $C_p(X)$ and we give the functional characterization of
the covering property of Hurewicz.

%Many equivalences hold among these properties, and the surviving
%ones appear in the following Diagram (Fig.~1) (where an arrow
%denotes implication), to which no arrow can be added except
%perhaps from $U_{fin}(\Gamma, \Gamma)$ or $U_{fin}(\Gamma,
%\Omega)$ to $S_{fin}(\Gamma, \Omega)$ \cite{jmss}.

\bigskip

%\begin{center}
%\ingrw{90}{Osipov_ris1.eps}

%\medskip

%Fig.~1. The Scheepers Diagram for Lindel$\ddot{o}$f space.

%\end{center}
%\bigskip

\section{Main definitions and notation}

Throughout this paper, all spaces are assumed to be Tychonoff. The
set of positive integers is denoted by $\mathbb{N}$. Let
$\mathbb{R}$ be the real line and $\mathbb{Q}$ be the rational
numbers. For a space $X$, we denote by $C_p(X)$ the space of all
real-valued continuous functions on $X$ with the topology of
pointwise convergence. The symbol ${\bf 0}$ denote the constantly
zero function in $C_p(X)$. Because $C_p(X)$ is homogeneous we can
work with ${\bf 0}$ to study local properties of $C_p(X)$.

Basic open sets of $C_p(X)$ are of the form

$[x_1,...,x_k, U_1,...,U_k]=\{f\in C(X): f(x_i)\in U_i$,
$i=1,...,k\}$, where each $x_i\in X$ and each $U_i$ is a non-empty
open subset of $\mathbb{R}$. Sometimes we will write the basic
neighborhood of the point $f$ as $\langle f,A,\epsilon \rangle$
where $\langle f,A,\epsilon \rangle:=\{g\in C(X):
|f(x)-g(x)|<\epsilon$ $\forall x\in A\}$, $A$ is a finite subset
of $X$ and $\epsilon>0$.

 If $X$ is a space and $A\subseteq X$, then the sequential closure of $A$,
 denoted by $[A]_{seq}$, is the set of all limits of sequences
 from $A$. A set $D\subseteq X$ is said to be sequentially dense
 if $X=[D]_{seq}$. A space $X$ is called sequentially separable if
 it has a countable sequentially dense set.

In this paper, by a cover we mean a nontrivial one, that is,
$\mathcal{U}$ is a cover of $X$ if $X=\bigcup \mathcal{U}$ and
$X\notin \mathcal{U}$.

 An open cover $\mathcal{U}$ of a space $X$ is:

 $\bullet$ an {\it $\omega$-cover} if every finite subset of $X$ is contained in a
 member of $\mathcal{U}$.

$\bullet$ a {\it $\gamma$-cover} if it is infinite and each $x\in
X$ belongs to all but finitely many elements of $\mathcal{U}$.
Note that every $\gamma$-cover contains a countably
$\gamma$-cover.

$\bullet$ a {\it $\gamma_F$-shrinkable} $\gamma$-cover
$\mathcal{U}$ if it is a $\gamma$-cover $\mathcal{U}$ of co-zero
sets of $X$ and there exists a $\gamma$-cover $\{F(U) : U\in
\mathcal{U}\}$ of zero-sets of $X$ with $F(U)\subset U$ for every
$U\in \mathcal{U}$.

For a topological space $X$ we denote:

$\bullet$ $\mathcal{O}$ --- the family of all open covers of $X$;

$\bullet$ $\Gamma$ --- the family of all countable open
$\gamma$-covers of $X$;

$\bullet$ $\Omega$ --- the family of all open $\omega$-covers of
$X$;

$\bullet$ $\Gamma_F$ --- the family of all $\gamma_F$-shrinkable
$\gamma$-covers of $X$.

For a topological space $C_p(X)$ we denote:

$\bullet$ $\mathcal{D}$ --- the family of all dense subsets of
$C_p(X)$;

$\bullet$ $\mathcal{S}$ --- the family of all sequentially dense
subsets of $C_p(X)$.

\medskip

In the case of $U_{fin}$ note that for any class of covers
$\mathcal{B}$ of Lindel$\ddot{o}$f space $X$,
$U_{fin}(\mathcal{O}, \mathcal{B})$ is equivalent to
$U_{fin}(\Gamma, \mathcal{B})$ because given an open cover $\{U_n
: n\in \mathbb{N}\}$ we may replace it by $\{\bigcup_{i<n}  U_i :
n\in \mathbb{N}\}$, which is a $\gamma$-cover (unless it contains
a finite subcover) of $X$.

\bigskip

 We recall that a subset of $X$ that is the
 complete preimage of zero for a certain function from~$C(X)$ is called a zero-set.
A subset $O\subseteq X$  is called  a cozero-set of $X$ if
$X\setminus O$ is a zero-set.

\medskip
Recall that the $i$-weight $iw(X)$ of a space $X$ is the smallest
infinite cardinal number $\tau$ such that $X$ can be mapped by a
one-to-one continuous mapping onto a Tychonoff space of the weight
not greater than $\tau$.

\begin{theorem} (Noble \cite{nob}) \label{th31} A space $C_{p}(X)$ is separable iff
$iw(X)=\aleph_0$.
\end{theorem}

\medskip

Let $X$ be a topological space, and $x\in X$. A subset $A$ of $X$
{\it converges} to $x$, $x=\lim A$, if $A$ is infinite, $x\notin
A$, and for each neighborhood $U$ of $x$, $A\setminus U$ is
finite. Consider the following collection:

$\bullet$ $\Omega_x=\{A\subseteq X : x\in \overline{A}\setminus
A\}$;

$\bullet$ $\Gamma_x=\{A\subseteq X : x=\lim A\}$.

Note that if $A\in \Gamma_x$, then there exists $\{a_n\}\subset A$
converging to $x$. So, simply $\Gamma_x$ may be the set of
non-trivial convergent sequences to $x$.

\bigskip

We write $\Pi (\mathcal{A}_x, \mathcal{B}_x)$ without specifying
$x$, we mean $(\forall x) \Pi (\mathcal{A}_x, \mathcal{B}_x)$.

 So we have three types of topological properties
described through the selection principles:

$\bullet$  local properties of the form $S_*(\Phi_x,\Psi_x)$;

$\bullet$  global properties of the form $S_*(\Phi,\Psi)$;

$\bullet$  semi-local of the form $S_*(\Phi,\Psi_x)$.

\section{$U_{fin}(\mathcal{O},\Omega)$ }

For a function space $C_p(X)$, we represent the next selection
principle

$F_{fin}(\mathcal{S},\mathcal{D})$:  whenever $\mathcal{S}_1$,
$\mathcal{S}_2, ... \in \mathcal{S}$, there are finite sets
$\mathcal{F}_n\subseteq \mathcal{S}_n$, $n\in \mathbb{N}$, such
that for each $f\in C_p(X)$ and a base neighborhood $\langle f, K,
\epsilon
 \rangle $ of $f$ where $\epsilon>0$ and $K=\{x_1, ..., x_k\}$ is a finite
subset of $X$, there is $n'\in \mathbb{N}$ such that for each
$j\in \{1,...,k\}$ there is $g\in \mathcal{F}_{n'}$ such that
$g(x_j)\in (f(x_j)-\epsilon, f(x_j)+\epsilon)$.

\medskip

It is clear that the condition of the selection principle
$F_{fin}(\mathcal{S},\mathcal{D})$ can be written more briefly:
whenever $\mathcal{S}_1$, $\mathcal{S}_2, ... \in \mathcal{S}$,
there are finite sets $\mathcal{F}_n\subseteq \mathcal{S}_n$,
$n\in \mathbb{N}$, such that for each $f\in C_p(X)$, $\epsilon>0$
and $K$ is a finite subset of $X$, there is $n'\in \mathbb{N}$
such that $\min\limits_{h\in \mathcal{F}_{n'}}
\{|f(x)-h(x)|\}<\epsilon$ for each $x\in K$.

Similarly, $F_{fin}(\Gamma_0,\Omega_0)$: whenever $\mathcal{S}_1$,
$\mathcal{S}_2, ... \in \Gamma_0$, there are finite sets
$\mathcal{F}_n\subseteq \mathcal{S}_n$, $n\in \mathbb{N}$, such
that for $\epsilon>0$ and $K$ is a finite subset of $X$, there is
$n'\in \mathbb{N}$ such that $\min\limits_{h\in \mathcal{F}_{n'}}
\{|h(x)|\}<\epsilon$ for each $x\in K$.

\begin{theorem}\label{th198} For a   space $X$, the following statements are
equivalent:

\begin{enumerate}

\item $C_p(X)$ satisfies $F_{fin}(\Gamma_{\bf0},\Omega_{\bf0})$;

\item $X$ satisfies $U_{fin}(\Gamma_F, \Omega)$.

\end{enumerate}

\end{theorem}

\begin{proof}

 $(1)\Rightarrow(2)$. Let $\{\mathcal{U}_i\}\subset \Gamma_F$, $\mathcal{U}_i=\{U^m_i\}$ for each $i\in
 \mathbb{N}$. We consider $\mathcal{K}_i =\{ f^m_i\in
C(X) : f^m_i\upharpoonright F(U^{m}_i)=0$ and
$f^m_i\upharpoonright (X\setminus U^{m}_i)=1$ for $m \in
\mathbb{N} \}$.

Since $\mathcal{F}_i=\{F(U^{m}_i): m\in \mathbb{N}\}$ is a
$\gamma$-cover of zero-sets of $X$, we have that $\mathcal{K}_i$
converge to $\bf{0}$ for each $i\in \mathbb{N}$. By $C_p(X)$
satisfies $F_{fin}(\Gamma_{\bf0},\Omega_{\bf0})$, there are finite
sets $F_i=\{f^{m_1}_i, ..., f^{m_{s(i)}}_i\}\subseteq
\mathcal{K}_i$ such that for a base neighborhood $O(f)=\langle f,
K, \epsilon
 \rangle $ of $f={\bf 0}$  where $\epsilon>0$
and $K=\{x_1, ..., x_k\}$ is a finite subset of $X$, there is
$n'\in \mathbb{N}$ such that for each $j\in \{1,...,k\}$ there is
$g\in \mathcal{F}_{n'}$ such that $g(x_j)\in (f(x_j)-\epsilon,
f(x_j)+\epsilon)$. Note that $\{\bigcup \{U^{m_1}_i, ...,
U^{m_{s(i)}}_i\}  : i\in \mathbb{N}\}\in \Omega$.

 $(2)\Rightarrow(1)$. Let $X$ satisfies $U_{fin}(\Gamma_F, \Omega)$
and $A_i\in \Gamma_{\bf{0}}$ for each $i\in \mathbb{N}$. Consider
$\mathcal{U}_i=\{U_{i,f}=f^{-1}(-\frac{1}{i}, \frac{1}{i}): f\in
A_i \}$ for each $i\in \mathbb{N}$. Without loss of generality we
can assume that a set $U_{i,f}\neq X$ for any $i\in \mathbb{N}$
and $f\in A_i$. Otherwise there is sequence $\{f_{i_k}\}_{k\in
\mathbb{N}}$ such that $\{f_{i_k}\}_{k\in \mathbb{N}}$ uniform
converge to $\bf{0}$ and $\{f_{i_k}: k\in \mathbb{N}\}\in
\Omega_{\bf 0}$.

Note that $\mathcal{F}_i=\{F_{i,m}\}_{m\in
\mathbb{N}}=\{f^{-1}_{i,m}[-\frac{1}{i+1}, \frac{1}{i+1}]: m\in
\mathbb{N} \}$ is $\gamma$-cover of zero-sets of $X$ and
$F_{i,m}\subset U_{i,m}$ for each $i,m\in \mathbb{N}$. It follows
that $\mathcal{U}_i\in \Gamma_F$ for each $i\in \mathbb{N}$.

By $X$ satisfies $U_{fin}(\Gamma_F, \Omega)$, there is a sequence
$\{U_{i,m(1)}, U_{i,m(2)}, ..., U_{i,m(i)}: i\in\mathbb{N}\}$ such
that for each $i$ and $k\in\{m(1),...,m(i)\}$, $U_{i,m(k)}\in
\mathcal{U}_i$, and

$\{\bigcup \{U_{i,m(1)}, ..., U_{i,m(i)}\}: i\in \mathbb{N}\}\in
\Omega$.

 Let $\langle {\bf 0}, K, \epsilon \rangle $ be a base neighborhood of
$\bf 0$ where $\epsilon>0$ and $K=\{x_1, ..., x_s\}$ is a finite
subset of $X$, then there are $i_0, i_1\in \mathbb{N}$ such that
$\frac{1}{i_0}<\epsilon$, $i_1>i_0$ and $\bigcup_{k=m(1)}^{m(i_1)}
U_{i_1,k}\supset K$. It follows that for each $j\in \{1,...,s\}$
there is $g\in \{f_{i_1,m(1)}, ..., f_{i_1,m(i_1)}\}$ such that
$g(x_j)\in (-\epsilon, \epsilon)$.

\end{proof}

\begin{lemma}(Lemma 6.5 in \cite{os2})\label{lemma} Let $\mathcal{U}=\{U_n:n\in \mathbb{N}\}$ be a
$\gamma_F$-shrinkable co-zero cover of a space $X$. Then the set
$S=\{f\in C(X): f\upharpoonright (X\setminus U_n)\equiv 1$ for
some $n\in \mathbb{N}\}$ is sequentially dense in $C_p(X)$.
\end{lemma}

\begin{theorem}\label{th195} For a space $X$ with $iw(X)=\aleph_0$, the following statements are equivalent:

\begin{enumerate}

\item $C_p(X)$ satisfies $F_{fin}(\mathcal{S},\mathcal{D})$;

\item $X$ satisfies $U_{fin}(\Gamma_F, \Omega)$;

\item $C_p(X)$ satisfies $F_{fin}(\Gamma_{\bf0}, \Omega_{\bf0})$;

\item $C_p(X)$ satisfies $F_{fin}(\mathcal{S}, \Omega_{\bf0})$.

\end{enumerate}

\end{theorem}

\begin{proof} $(1)\Rightarrow(2)$. Let $\mathcal{U}_i=\{U^j_i: j\in \mathbb{N}\}\in \Gamma_F$ for
each $i\in \mathbb{N}$. Then, by Lemma \ref{lemma}, each
$S_i=\{f\in C(X): f\upharpoonright(X\setminus U^j_i)\equiv 1$ for
some $m\in \mathbb{N}\}$ is sequentially dense in $C_p(X)$. By
$C_p(X)$ satisfies $F_{fin}(\mathcal{S},\mathcal{D})$, there are
finite sets $F_i=\{f^{m_1}_i, ..., f^{m_{s(i)}}_i\}\subseteq
\mathcal{S}_i$ such that for each $f\in C_p(X)$ and a base
neighborhood $\langle f, K, \epsilon
 \rangle $ of $f$ where $\epsilon>0$ and $K=\{x_1, ..., x_k\}$ is a finite
subset of $X$, there is $n'\in \mathbb{N}$ such that for each
$j\in \{1,...,k\}$ there is $g\in F_{n'}$ such that $g(x_j)\in
(f(x_j)-\epsilon, f(x_j)+\epsilon)$. Note that $\{\bigcup
\{U^{m_1}_i, ..., U^{m_{s(i)}}_i\} : i\in \mathbb{N}\}\in \Omega$.

$(2)\Rightarrow(3)$. By  Theorem \ref{th198}.

$(3)\Rightarrow(4)$ is immediate.

$(4)\Rightarrow(1)$. Suppose that $C_p(X)$ satisfies
$F_{fin}(\mathcal{S}, \Omega_{\bf0})$.

Let $D=\{d_n: n\in \mathbb{N} \}$ be a dense subspace of $C_p(X)$
and  $S_i\in \mathcal{S}$ for each $i\in \mathbb{N}$. Given a
sequence of sequentially dense subspace of $C_p(X)$, enumerate it
as $\{S_{n,m}: n,m \in \mathbb{N} \}$. For each $n\in \mathbb{N}$,
pick

$\mathcal{F}_{n,m}=\{d_{n,m,1},..., d_{n,m,k(n,m)} \} \subset
S_{n,m}$ so that for a base neighborhood $\langle d_n, K, \epsilon
 \rangle $ of $d_n$  where $\epsilon>0$
and $K=\{x_1, ..., x_k\}$ is a finite subset of $X$, there is
$m'\in \mathbb{N}$ such that for each $j\in \{1,...,k\}$ there is
$g\in \mathcal{F}_{n,m'}$ such that $g(x_j)\in (d_n(x_j)-\epsilon,
d_n(x_j)+\epsilon)$. It follows that $C_p(X)$ satisfies
$F_{fin}(\mathcal{S},\mathcal{D})$.

\end{proof}

\begin{theorem} For a Tychonoff space $X$ the
following statements are equivalent:

\begin{enumerate}

\item $X$ is Lindel$\ddot{o}$f and $X$ satisfies
$U_{fin}(\Gamma_F, \Omega)$;

\item $X$ satisfies $U_{fin}(\mathcal{O}, \Omega)$.

\end{enumerate}

\end{theorem}

\begin{proof} It is proved similarly to the proof of
Theorem \ref{thur}.

\end{proof}

\begin{theorem}\label{th1951} For a separable metrizable space $X$, the following statements are
equivalent:

\begin{enumerate}

\item $C_p(X)$ satisfies $F_{fin}(\mathcal{S},\mathcal{D})$;

\item $X$ satisfies $U_{fin}(\mathcal{O}, \Omega)$;

\item $C_p(X)$ satisfies $F_{fin}(\Gamma_{\bf0}, \Omega_{\bf0})$;

\item $C_p(X)$ satisfies $F_{fin}(\mathcal{S}, \Omega_{\bf0})$.

\end{enumerate}

\end{theorem}

\section{$U_{fin}(\mathcal{O}, \Gamma)$ - Hurewicz property}

In \cite{hur1} (see also \cite{hur2}), Hurewicz introduced a
covering property of a space $X$, nowadays called the {\it
Hurewicz property} in this way: for each sequence $(U_n : n\in
\mathbb{N})$ of open covers of $X$ there is a sequence $(V_n :
n\in \mathbb{N})$ such that for each $n\in \mathbb{N}$, $V_n$ is a
finite subset of $U_n$ and each $x\in X$ belongs to $\cup V_n$ for
all but finitely many $n$ (i.e., $X$ satisfies
$U_{fin}(\mathcal{O}, \Gamma)$).

\begin{theorem}\label{thur} For a Tychonoff space $X$ the
following statements are equivalent:

\begin{enumerate}

\item $X$ is Lindel$\ddot{o}$f and satisfies $U_{fin}(\Gamma_F,
\Gamma)$;

\item $X$ has the Hurewicz property.

\end{enumerate}

\end{theorem}

\begin{proof} $(1)\Rightarrow(2)$. Let $(\mathcal{U}_n : n\in
\mathbb{N})$ be a sequence of open covers of $X$. For every $n$,
$U\in \mathcal{U}_n$ and $x\in X$ we find co-zero sets
$W_{0,n,U,x}$ and $W_{2,n,U,x}$, and, a zero-set $W_{1,n,U,x}$
such that $x\in W_{0,n,U,x}\subset W_{1,n,U,x}\subset
W_{2,n,U,x}\subset U$. Since $X$ is Lindel$\ddot{o}$f, there is a
sequence $(x^n_k : k\in \mathbb{N})$ such that $X$ is covered be
$\{W_{0,n,U,x^n_k} : k\in \mathbb{N}\}$. Look at the cover
$\mathcal{W}_n$ of $X$ consisting of sets $W^n_k=\bigcup_{i\leq k}
W_{2,n,U,x^n_{i}}$, $k\in \mathbb{N}$. Note that $\mathcal{W}_n\in
\Gamma_F$ because $\bigcup_{i\leq k} W_{1,n,U,x^n_{i}}$ is a
zero-set contained in $W^n_k$, and $\{\bigcup_{i\leq k}
W_{1,n,U,x^n_{i}} : k\in \mathbb{N}\}$ is a $\gamma$-cover of $X$
because even $\{\bigcup_{i\leq k} W_{0,n,U,x^n_{i}} : k\in
\mathbb{N}\}$ is a $\gamma$-cover of $X$.

Now use the property $U_{fin}(\Gamma_F, \Gamma)$ to the sequence
$(\mathcal{W}_n : n\in \mathbb{N})$ together with the fact that
$\mathcal{W}_n$ is a finer cover that $\mathcal{U}_n$ for all $n$.

\end{proof}

For a function space $C_p(X)$, we represent the next selection
principle  $F_{fin}(\mathcal{S},\mathcal{S})$:

 whenever $\mathcal{S}_1$, $\mathcal{S}_2, ... \in \mathcal{S}$, there are finite sets $\mathcal{F}_n\subseteq \mathcal{S}_n$,
$n\in \mathbb{N}$, such that for each $f\in C_p(X)$  there is
$\{\mathcal{F}_{n_k} : k\in \mathbb{N}\}$ such that for a base
neighborhood $\langle f, K, \epsilon
 \rangle $ of $f$  where $\epsilon>0$ and $K=\{x_1, ..., x_m\}$ is a
finite subset of $X$, there is $k'\in \mathbb{N}$ such that for
each $k>k'$ and $\forall$ $j\in \{1,...,m\}$ there is $g\in
\mathcal{F}_{n_k}$ such that $g(x_j)\in (f(x_j)-\epsilon,
f(x_j)+\epsilon)$.

It is clear that the condition of the selection principle
$F_{fin}(\mathcal{S},\mathcal{S})$ can be written more briefly:
whenever $\mathcal{S}_1$, $\mathcal{S}_2, ... \in \mathcal{S}$,
there are finite sets $\mathcal{F}_n\subseteq \mathcal{S}_n$,
$n\in \mathbb{N}$, such that for each $f\in C_p(X)$, $\epsilon>0$
and $K$ is a finite subset of $X$, there is $n'\in \mathbb{N}$
such that for every $n>n'$ $\min\limits_{h\in \mathcal{F}_{n}}
\{|f(x)-h(x)|\}<\epsilon$ for each $x\in K$.

Similarly, $F_{fin}(\Gamma_0,\Gamma_0)$: whenever $\mathcal{S}_1$,
$\mathcal{S}_2, ... \in \Gamma_0$, there are finite sets
$\mathcal{F}_n\subseteq \mathcal{S}_n$, $n\in \mathbb{N}$, such
that for $\epsilon>0$ and $K$ is a finite subset of $X$, there is
$n'\in \mathbb{N}$ such that for every $n>n'$ $\min\limits_{h\in
\mathcal{F}_{n}} \{|h(x)|\}<\epsilon$ for each $x\in K$.

\begin{theorem}\label{th194} For a space $X$, the following statements are
equivalent:

\begin{enumerate}

\item $C_p(X)$ satisfies $F_{fin}(\Gamma_{\bf0},\Gamma_{\bf0})$;

\item $X$ satisfies $U_{fin}(\Gamma_{F}, \Gamma)$.

\end{enumerate}

\end{theorem}

\begin{proof} $(1)\Rightarrow(2)$.  Let $\{\mathcal{U}_i\}\subset \Gamma_F$, $\mathcal{U}_i=\{
U^{m}_i\}_{m\in \mathbb{N}}$ for each $i\in \mathbb{N}$. We
consider a
 subset $\mathcal{S}_i$ of $C_p(X)$  where

$\mathcal{S}_i =\{ f^m_i\in C(X) : f^m_i\upharpoonright
F(U^{m}_i)=0$ and $f^m_i\upharpoonright (X\setminus U^{m}_i)=1$
for $m \in \mathbb{N} \}$.

Since $\mathcal{F}_i=\{F(U^{m}_i): m\in \mathbb{N}\}$ is a
$\gamma$-cover of $X$, we have that $\mathcal{S}_i$ converge to
${\bf 0}$, i.e. $\mathcal{S}_i\in \Gamma_0$ for each $i\in
\mathbb{N}$.

Since  $C(X)$ satisfies $F_{fin}(\Gamma_{\bf0},\Gamma_{\bf0})$,
there is a sequence $\{\mathcal{F}_i\}_{i\in
\mathbb{N}}=\{f^{m_1}_{i},..., f^{m_{k(i)}}_{i} :
i\in\mathbb{N}\}$ such that for each $i$, $\mathcal{F}_i\subseteq
\mathcal{S}_i$, and for a base neighborhood $\langle f, K,
\epsilon
 \rangle $ of $f={\bf 0}$  where $\epsilon>0$
and $K=\{x_1, ..., x_k\}$ is a finite subset of $X$, there is
$n'\in \mathbb{N}$ such that for every $n>n'$ and $\forall$ $j\in
\{1,...,k\}$ there is $g\in \mathcal{F}_{n}$ such that $g(x_j)\in
(f(x_j)-\epsilon, f(x_j)+\epsilon)$.

Consider the sequence $\{W_i\}_{i\in \mathbb{N}}=\{U^{m_1}_{i},
...,U^{m_{k(i)}}_{i} : i\in \mathbb{N}\}$.

(a). $W_i \subset \mathcal{U}_{i}$.

(b). $\{\bigcup W_i: i\in \mathbb{N}\}$ is a $\gamma$-cover of
$X$.

 Let $K=\{x_1,...,x_s\}$ be a finite subset of $X$ and $\langle {\bf 0},
K, \frac{1}{2} \rangle $ be a base neighborhood of ${\bf 0}$, then
there exists $i_0\in \mathbb{N}$ such that for each $i>i_0$ and

$j\in \{1,...,s\}$ there is $g\in \mathcal{F}_{i}$ such that
$g(x_j)\in (-\frac{1}{2}, \frac{1}{2})$.

 It follows that
$K\subset \bigcup\limits_{j=1}^{k(i)} U^{m_j}_{i}$ for $i>i_0$. We
thus get $X$ satisfies $U_{fin}(\Gamma_{F}, \Gamma)$.

$(2)\Rightarrow(1)$.  Fix $\{S_i: i\in \mathbb{N}\}\subset
\Gamma_0$ where $S_i=\{f^{i}_k: k\in \mathbb{N}\}$ for each $i\in
\mathbb{N}$.

For each $i,k\in \mathbb{N}$, we put $U_{i,k}=\{x\in X :
|f^{i}_k(x)|<\frac{1}{i}\}$, $Z_{i,k}=\{x\in X :
|f^{i}_k(x)|\leq\frac{1}{i+1}\}$.

Each $U_{i,k}$ (resp., $Z_{i,k}$) is a cozero-set (resp.,
zero-set) in $X$ with $Z_{i,k}\subset U_{i,k}$. Let
$\mathcal{U}_i=\{ U_{i,k} : k\in \mathbb{N}\}$ and let
$\mathcal{Z}_i=\{ Z_{i,k} : k\in \mathbb{N}\}$. So without loss of
generality, we may assume $U_{i,k}\neq X$ for each $i,k\in
\mathbb{N}$. We can easily check that the condition $f^{i}_k
\rightarrow \bf{0}$ ($k \rightarrow \infty $) implies that
$\mathcal{Z}_i$ is a $\gamma$-cover of $X$.

Since  $X$ satisfies $U_{fin}(\Gamma_{F}, \Gamma)$ there is a
sequence $\{\mathcal{F}_i\}_{i\in
\mathbb{N}}=(U_{i,k_1},...,U_{i,k_i}: i\in\mathbb{N})$ such that
for each $i$, $\mathcal{F}_i \subset \mathcal{U}_i$, and
$\{\bigcup \mathcal{F}_i: i\in\mathbb{N} \}$ is an element of
$\Gamma$.

Let $K=\{x_1,...,x_s\}$ be a finite subset of $X$, $\epsilon>0$,
and $\langle {\bf 0}, K, \epsilon \rangle $ be a base neighborhood
of ${\bf 0}$, then there exists $i'\in \mathbb{N}$ such that for
every $i>i'$ $K\subset \bigcup \mathcal{F}_i$. It follow that for
every $i>i'$ and $j\in \{1,...,s\}$ there is $g\in
\mathcal{S}_{i}$ such that $g(x_j)\in (-\epsilon, \epsilon)$. So
$C_p(X)$ satisfies $F_{fin}(\Gamma_{\bf0},\Gamma_{\bf 0})$.

\end{proof}

\begin{theorem}\label{th199} For a Lindel$\ddot{o}$f space $X$, the following statements are
equivalent:

\begin{enumerate}

\item $C_p(X)$ satisfies $F_{fin}(\Gamma_{\bf0},\Gamma_{\bf0})$;

\item $X$ has the Hurewicz property.

\end{enumerate}

\end{theorem}

 A space $X$ has {\it $Velichko$ property} ($X$ $\models$ $V$), if there
 exist  a condensation (one-to-one continuous mapping) $f: X \mapsto Y$ from the space $X$ on a
 separable metric space $Y$, such that $f(U)$ is an $F_{\sigma}$-set
 of $Y$ for any cozero-set $U$ of $X$.

\begin{theorem} \label{th38} (Velichko \cite{vel}). Let $X$ be a space. A space $C_p(X)$ is
sequentially separable iff  $X$ $\models$ $V$.
\end{theorem}

\medskip

\begin{theorem}\label{th153} For a space $X$ with $X$ $\models$
$V$, the following statements are equivalent:

\begin{enumerate}

\item $C_p(X)$ satisfies $F_{fin}(\mathcal{S},\mathcal{S})$;

\item $X$ satisfies $U_{fin}(\Gamma_F, \Gamma)$;

\item  $C_p(X)$ satisfies $F_{fin}(\Gamma_{\bf0}, \Gamma_{\bf0})$;

\item  $C_p(X)$ satisfies $F_{fin}(\mathcal{S}, \Gamma_{\bf0})$.

\end{enumerate}

\end{theorem}

\begin{proof} $(1)\Rightarrow(2)$. Let $\mathcal{U}_i=\{U^j_i: j\in \mathbb{N}\}\in \Gamma_F$ for
each $i\in \mathbb{N}$. Then, by Lemma \ref{lemma}, each
$S_i=\{f\in C(X): f\upharpoonright(X\setminus U^j_i)\equiv 1$ for
some $m\in \mathbb{N}\}$ is sequentially dense in $C_p(X)$.

Since  $C(X)$ satisfies $U_{fin}(\mathcal{S},\mathcal{S})$, there
is a sequence $\{\mathcal{F}_i\}=\{f^{m_1}_{i}, ...,f^{m_s}_{i} :
i\in\mathbb{N}\}$ such that for each $i$, $\mathcal{F}_i \subset
\mathcal{S}_i$, and for each $f\in C_p(X)$  there is
$\{\mathcal{F}_{n_k} : k\in \mathbb{N}\}$ such that for a base
neighborhood $\langle f, K, \epsilon
 \rangle $ of $f$  where $\epsilon>0$ and $K=\{x_1, ..., x_m\}$ is a
finite subset of $X$, there is $k'\in \mathbb{N}$ such that for
each $k>k'$ and $\forall$ $j\in \{1,...,m\}$ there is $g\in
\mathcal{F}_{n_k}$ such that $g(x_j)\in (f(x_j)-\epsilon,
f(x_j)+\epsilon)$.

Then for $f={\bf 0}$ there is $\{F_{i_k} : k\in \mathbb{N}\}$ such
that for a base neighborhood $\langle f, K, \epsilon
 \rangle $ of $f$  where $\epsilon>0$ and $K=\{x_1, ..., x_m\}$ is a
finite subset of $X$, there is $k'\in \mathbb{N}$ such that for
each $k>k'$ and $\forall$ $j\in \{1,...,m\}$ there is $g\in
F_{i_k}$ such that $g(x_j)\in (-\epsilon, \epsilon)$.

Let $\epsilon=\frac{1}{2}$. Consider a sequence $\{Q_k\}_{k\in
\mathbb{N}\setminus\{k'\}}=\{U^{m_1}_{i_k},...,U^{m_s}_{i_k}: k\in
\mathbb{N}\setminus\{k'\} \}$ for corresponding to
$\{F_{i_k}\}=\{f^{m_1}_{i_k}, ...,f^{m_s}_{i_k} : k\in
\mathbb{N}\setminus\{k'\} \}$.

(a). $Q_k \subset \mathcal{U}_{i_k}$ for $k\in
\mathbb{N}\setminus\{k'\}$.

(b). $\{\bigcup Q_k: k\in \mathbb{N}\setminus\{k'\} \}$ is a
$\gamma$-cover of $X$. We thus get $X$ satisfies
$U_{fin}(\Gamma_F, \Gamma)$.

$(2)\Rightarrow(3)$. By Theorem \ref{th194}.

$(3)\Rightarrow(4)$ is immediate.

$(4)\Rightarrow(1)$. For each $n\in \mathbb{N}$, let $S_n$ be a
sequentially dense subset of $C_p(X)$ and let $\{h_n: n\in
\mathbb{N}\}$ be sequentially dense in $C_p(X)$. Take a sequence
$\{f^m_n: m\in \mathbb{N}\}\subset S_n$ such that $f^m_n\mapsto
h_n$ ($m\mapsto \infty$). Then $f^m_n-h_n\mapsto \bf{0}$
($m\mapsto \infty$). Hence, there exist
$\mathcal{F}_n=\{f^{m_1}_n, ...,f^{m_{k(n)}}_n\}\subset S_n$ such
that $\{\bigcup \{f^{m_1}_n-h_n, ...,f^{m_{k(n)}}_n-h_n\} : n\in
\mathbb{N}\}\in \Gamma_0$, i.e. for a base neighborhood $\langle
f, K, \epsilon
 \rangle $ of $f={\bf 0}$  where $\epsilon>0$ and $K=\{x_1, ..., x_m\}$ is a finite subset of $X$, there is
$n'\in \mathbb{N}$ such that for each $n>n'$ and $\forall$ $j\in
\{1,...,m\}$ there is $g\in \{f^{m_1}_n-h_n,
...,f^{m_{k(n)}}_n-h_n\}$ such that $g(x_j)\in (-\epsilon,
\epsilon)$.

 Let $h\in C_p(X)$ and take a sequence $\{h_{n_j}:
j\in \mathbb{N}\}\subset \{h_n:n\in \mathbb{N}\}$ converging to
$h$. Let $K=\{x_1,...,x_m\}$ be a finite subset of $X$ and
$\epsilon>0$. Consider a base neighborhood $\langle h, K, \epsilon
 \rangle $ of $h$. Then there is $j'\in \mathbb{N}$ such that $h_{n_j}\in \langle h, K,
\frac{\epsilon}{2} \rangle $ and  $\forall$ $s\in \{1,...,m\}$
there is $g\in \{f^{m_1}_{n_j}-h_{n_j},
...,f^{m_{k(n_{j})}}_{n_j}-h_{n_j}\}$ such that $g(x_s)\in
(-\frac{\epsilon}{2}, \frac{\epsilon}{2})$ for $j>j'$. It follows
that for each $s\in \{1,...,m\}$ there is $l(j)\in
\overline{1,k(n_{j})}$ such that
$((f^{m_{l(j)}}_{n_j}-h_{n_j})+(h_{n_j}-h))(x_s) \in (-\epsilon,
\epsilon)$ for $j>j'$. Hence $C_p(X)$ satisfies
$F_{fin}(\mathcal{S},\mathcal{S})$.

\end{proof}

\begin{theorem}\label{th1531} For a separable metrizable space $X$, the following statements are
equivalent:

\begin{enumerate}

\item $C_p(X)$ satisfies $F_{fin}(\mathcal{S},\mathcal{S})$;

\item $X$ satisfies $U_{fin}(\mathcal{O}, \Gamma)$ [Hurewicz
property];

\item  $C_p(X)$ satisfies $F_{fin}(\Gamma_{\bf0}, \Gamma_{\bf0})$;

\item  $C_p(X)$ satisfies $F_{fin}(\mathcal{S}, \Gamma_{\bf0})$.

\end{enumerate}

\end{theorem}

\medskip

\medskip
Recall that a space $X$ is said to be Rothberger \cite{rot} (or,
\cite{mil}) if for every sequence $(\mathcal{U}_n : n\in \omega)$
of open covers of $X$, there is a sequence $(V_n : n\in
\mathbb{N})$ such that for each $n$, $V_n\in \mathcal{U}_n$, and
$\{V_n : n\in\mathbb{N} \}$ is an open cover of $X$.

\medskip
A space $X$ is said to be Menger \cite{hur1} (or, \cite{sash}) if
for every sequence $(\mathcal{U}_n : n\in \mathbb{N})$ of open
covers of $X$, there are finite subfamilies $\mathcal{V}_n\subset
\mathcal{U}_n$ such that $\bigcup \{\mathcal{V}_n : n\in
\mathbb{N} \}$ is a cover of $X$.

 Every $\sigma$-compact space is Menger,
and a Menger space is Lindel$\ddot{o}$f.

\medskip
In \cite{os4}, we gave the functional characterizations of the
Rothberger and Menger properties.

\medskip

Recall that if $C_p(X)$ and $C_p(Y)$ are homeomorphic (linearly
homeomorphic, uniformly homeomorphic), we say that the spaces $X$
and $Y$ are $t$-equivalent ($l$-equivalent, $u$-equivalent). The
properties preserved by $t$-equivalence ($l$-equivalence,
$u$-equivalence) we call $t$-invariant ($l$-invariant,
$u$-invariant) \cite{arh2}.

\medskip

{\bf Question 1.} Is the Hurewicz (Rothberger, Menger)  property
$t$-invariant? $l$-invariant?  $u$-invariant?

\medskip

%{\bf Acknowledgment.} I should like to Thanks to Boaz Tsaban for
%useful discussions and recommendation to publish this research.

%\bigskip

\bibliographystyle{model1a-num-names}
\bibliography{<your-bib-database>}


%%\bibliographystyle{plain}




% \end{thebibliography}

\end{document}